\documentclass[onefignum,onetabnum]{siamonline181217}

\usepackage{lipsum}
\usepackage{amsfonts}
\usepackage{graphicx}
\usepackage{epstopdf}
\usepackage{algorithmic}
\ifpdf
  \DeclareGraphicsExtensions{.eps,.pdf,.png,.jpg}
\else
  \DeclareGraphicsExtensions{.eps}
\fi

\usepackage{enumitem}
\setlist[enumerate]{leftmargin=.5in}
\setlist[itemize]{leftmargin=.5in}

\newsiamremark{remark}{Remark}
\newsiamremark{hypothesis}{Hypothesis}
\crefname{hypothesis}{Hypothesis}{Hypotheses}
\newsiamthm{claim}{Claim}

\headers{Permanence of Mass-Action Systems}{B. Boros and J. Hofbauer}

\title{Permanence of Weakly Reversible Mass-Action Systems\\ with a Single Linkage Class\thanks{
\funding{The work of B. Boros was supported by the Austrian Science Fund (FWF), project
P28406.}}}

\author{Bal\'azs Boros\thanks{Faculty of Mathematics, University of Vienna, Vienna, Austria 
  (\email{balazs.boros@univie.ac.at}, \url{https://web.cs.elte.hu/\string~bboros/}).}
\and Josef Hofbauer\thanks{Faculty of Mathematics, University of Vienna, Vienna, Austria 
  (\email{josef.hofbauer@univie.ac.at}, \email{https://homepage.univie.ac.at/josef.hofbauer/}).}}

\usepackage{amsopn}

\DeclareMathOperator{\grad}{grad}
\newcommand*\diff{\mathop{}\!\mathrm{d}}
\usepackage{tikz}
\usetikzlibrary{arrows}
\usetikzlibrary{arrows.meta}
\crefname{lemma}{Lemma}{Lemmata}
\crefname{theorem}{Theorem}{Theorems}
\Crefname{lemma}{L}{Ls}   
\Crefname{theorem}{T}{Ts} 

\ifpdf
\hypersetup{
  pdftitle={Permanence of Weakly Reversible Mass-Action Systems with a Single Linkage Class},
  pdfauthor={B. Boros and J. Hofbauer}
}
\fi

\begin{document}

\maketitle

\begin{abstract}
We give a new proof of the fact that each weakly reversible mass-action system with a single linkage class is permanent.
\end{abstract}



\section{Introduction}

Horn and Jackson \cite{horn:jackson:1972} considered the class of weakly reversible mass-action systems and they showed that complex balanced equilibria are unique and asymptotically stable relative to their positive stoichiometric class by providing a global Lyapunov function. However, existence of complex balanced equilibria is a rare situation. On the other hand, every weakly reversible mass-action system has a positive equilibrium in each positive stoichiometric class \cite{boros:2019}. In general, little can be said about the number of equilibria and their stability. It is natural to ask whether other (more global) dynamical properties hold for weakly reversible mass-action systems. For example whether all species coexist for all times and their concentrations remain bounded. These ideas have been formalized as the concepts of persistence, boundedness, and permanence, and several conjectures have been formulated for weakly reversible (and even more general classes, like endotactic) mass-action systems, see e.g.\ \cite{craciun:dickenstein:shiu:sturmfels:2009}, \cite{anderson:2011a}, \cite{anderson:2011b}, \cite{craciun:nazarov:pantea:2013}, \cite{pantea:2012}, \cite{gopalkrishnan:miller:shiu:2013}, \cite{gopalkrishnan:miller:shiu:2014}, and \cite{anderson:cappelletti:kim:nguyen:2018}. Many of these papers use the classical, entropy-like Lyapunov function by Horn and Jackson. In particular, Anderson shows persistence and boundedness for weakly reversible mass-action systems with a single linkage class (see \cite{anderson:2011a} and \cite{anderson:2011b}), Gopalkrishnan, Miller, and Shiu extend this to show permanence for the even larger class of strongly endotactic mass-action systems (see \cite{gopalkrishnan:miller:shiu:2013} and \cite{gopalkrishnan:miller:shiu:2014}), and Anderson, Cappelletti, Kim, and Nguyen give further insight (see \cite{anderson:cappelletti:kim:nguyen:2018}).

In the present paper, we provide a new proof for the permanence of weakly reversible mass-action systems with a single linkage class. We simplify the previous approaches, and give a more concise and elementary proof of the fact that a variant of the classical Horn-Jackson function is a Lyapunov function near the boundary of the state space. We crucially use ideas from \cite{deng:feinberg:jones:nachman:2011} and \cite{boros:2019}. On the other hand, we demonstrate in \cref{sec:examples} that the method breaks down for weakly reversible networks with multiple linkage classes.

The rest of this paper is organized as follows. After introducing some notations and the necessary notions from chemical reaction network theory (CRNT) in \cref{sec:notations,sec:crnt}, respectively, we state the main result of this paper in \cref{sec:main}. We collect a series of lemmata in \cref{sec:lemmata} and prove the main result in \cref{sec:proof_main}. We illustrate the limitations of the method in \cref{sec:examples}. Finally, we display in \cref{sec:acyclic} the dependency of the statements in \cref{sec:main,sec:lemmata} via an acyclic digraph.

\section{Notations} \label{sec:notations}

We use standard notations.

The symbols $\mathbb{R}_+$ and $\mathbb{R}_{\geq0}$ denote the sets of positive and nonnegative real numbers, respectively. Accordingly, $\mathbb{R}_+^n$ and $\mathbb{R}_{\geq0}^n$ denote the positive and nonnegative orthants, respectively.

For the vectors $x\in\mathbb{R}^n_+$ and $y\in\mathbb{R}^n$, the notation $x^y$ is a shorthand for the monomial $\prod_{s=1}^nx_s^{y_s}$.

For a vector $z \in \mathbb{R}^m$, the smallest and largest value among the coordinates of $z$ are denoted by $\min z$ and $\max z$, respectively.

For a vector $x \in \mathbb{R}^n$ and a subset $W$ of $\{1,\ldots,n\}$, the vector $x_W\in\mathbb{R}^W$ is the restriction of $x$ to the indices in $W$.

For a set $W\subseteq\{1,\ldots,n\}$, we denote by $W^c$ the set $\{1,\ldots,n\}\setminus W$.

For a set $A\subseteq \mathbb{R}^n$, the symbols $\overline{A}$ and $\partial A$ denote the closure and the boundary of $A$, respectively.

\section{Mass-action systems} \label{sec:crnt}

We give a very brief introduction to the basic notions of CRNT. For more details, the reader is advised to consult e.g. \cite{feinberg:1979}, \cite{feinberg:1987}, and \cite{gunawardena:2003}.

A \emph{reaction network} is a triple $(\mathcal{X},\mathcal{C},\mathcal{R})$, where $\mathcal{X}$, $\mathcal{C}$, and $\mathcal{R}$ are the set of \emph{species}, \emph{complexes}, and \emph{reactions}, respectively. Throughout the paper, we use $n = |\mathcal{X}|$ and $m = |\mathcal{C}|$. The complexes are formal linear combinations of the species, the coefficients are stored in the matrix $Y = [y_1,\ldots,y_m] \in \mathbb{R}^{n \times m}$. The $i$th complex is then $(y_i)_1\mathsf{X}_1 +\cdots + (y_i)_n\mathsf{X}_n$, where $\mathsf{X}_1, \ldots,\mathsf{X}_n$ denote the species. The set $\mathcal{R}$ consists of ordered pairs of complexes, the first and the second element of the pair are called \emph{reactant} complex and \emph{product} complex, respectively.

The weak components of the digraph $(\mathcal{C},\mathcal{R})$ are called \emph{linkage classes}. The reaction network is said to be \emph{weakly reversible} if all the weak components of the digraph $(\mathcal{C},\mathcal{R})$ are strongly connected, i.e., for every pair $(i,j)$ of complexes, the existence of a directed path from $i$ to $j$ implies the existence of a directed path from $j$ to $i$. Most of the results in this paper are about weakly reversible networks with a single linkage class, i.e., the digraph $(\mathcal{C},\mathcal{R})$ is strongly connected. When having a single linkage class network, usually it is required that there are at least two complexes and all the complexes are distinct. However, we assume only that there exist at least two complexes that are distinct (thus, there might be $i,j\in\{1,\ldots,m\}$ such that $i\neq j$ and $y_i=y_j$). We make this weakening for convenience.

Denoting by $x(\tau) \in \mathbb{R}^n_+$ the \emph{concentration vector} of the species at time $\tau$, assuming \emph{mass-action kinetics}, the time evolution of the species concentration vector is described by the autonomous ordinary differential equation (ODE)
\begin{align*}
\dot{x}(\tau) = \sum_{(i,j)\in\mathcal{R}} \kappa_{ij}x(\tau)^{y_i}(y_j - y_i) \text{ with state space } \mathbb{R}^n_+,
\end{align*}
where $\kappa : \mathcal{R} \to \mathbb{R}_+$. The reason we defined the ODE in the positive orthant (not in the nonnegative orthant) is that we allow negative entries in $Y$. The quadruple $(\mathcal{X},\mathcal{C},\mathcal{R},\kappa)$ is called a \emph{mass-action system}.

In fact, we will consider a slightly more general family of differential equations. Namely, the non-autonomous ODE
\begin{align} \label{eq:ODE}
\dot{x}(\tau) = \sum_{(i,j)\in\mathcal{R}} \kappa_{ij}(\tau)x(\tau)^{y_i}(y_j - y_i) \text{ with state space } \mathbb{R}^n_+,
\end{align}
with an $0<\varepsilon<1$ such that $\varepsilon\leq\kappa_{ij}(\tau)\leq 1/\varepsilon$ for all $(i,j)\in\mathcal{R}$ and for all $\tau\geq0$. The quadruple $(\mathcal{X},\mathcal{C},\mathcal{R},\kappa(\tau))$ is called a \emph{mass-action system with bounded kinetics}.

The linear subspace $\mathcal{S}$ of $\mathbb{R}^n_+$, spanned by the set $\{y_i-y_j~|~(i,j)\in\mathcal{R}\}$, is called the \emph{stoichiometric subspace}. Clearly, the translations of $\mathcal{S}$ are forward invariant under the ODE \eqref{eq:ODE}. The sets $(p+\mathcal{S}) \cap \mathbb{R}^n_+$ for $p \in \mathbb{R}^n_+$ are called \emph{positive stoichiometric classes}. The relevant approach is to study the dynamics within a fixed positive stoichometric class.

\section{Main result}
\label{sec:main}

The formal definition of permanence is the following.

\begin{definition}
A mass-action system with bounded kinetics is said to be \emph{permanent} on the positive stoichiometric class $\mathcal{P}$ if there exists a compact set $K\subseteq \mathcal{P}$ with the property that for each solution $\tau\mapsto x(\tau)$ with $x(0)\in\mathcal{P}$, there exists a $\tau_0\geq0$ such that for all $\tau\geq\tau_0$ we have $x(\tau)\in K$.
\end{definition}

Results on permanence of planar systems include \cite{simon:1995}, \cite{craciun:nazarov:pantea:2013}, \cite{brunner:craciun:2018}, and \cite{boros:hofbauer:2019}, while in the papers \cite{anderson:2011b} and \cite{gopalkrishnan:miller:shiu:2014} there is no restriction on the number of species. The \emph{Permanence Conjecture} states that all weakly reversible mass-action systems with bounded kinetics are permanent in each positive stoichiometric class. In \cite{gopalkrishnan:miller:shiu:2013} and \cite{gopalkrishnan:miller:shiu:2014}, the conjecture is proven in the single linkage class case, see \cite[Theorem 1.3]{gopalkrishnan:miller:shiu:2014}, and the same result follows from \cite[Theorem 5.5]{anderson:cappelletti:kim:nguyen:2018} in case the rate constants do not depend on time. In this paper, we give a more concise proof of the single linkage class case of the conjecture.

Next, we state our main result as \cref{thm:main}. Its proof is carried out in \cref{sec:proof_main}, after collecting a series of lemmata in \cref{sec:lemmata}.

\begin{theorem} \label{thm:main}
Let $(\mathcal{X},\mathcal{C},\mathcal{R},\kappa(\tau))$ be a weakly reversible, single linkage class mass-action system with bounded kinetics.  Fix a positive stoichiometric class $\mathcal{P}$. Then
\begin{itemize}
\item[(i)] there exists a compact and forward invariant set $K\subseteq \mathcal{P}$ with the property that each solution starting in $\mathcal{P}$ enters $K$ in finite time and
\item[(ii)] the system, restricted to $\mathcal{P}$, is permanent.
\end{itemize}
\end{theorem}

By forward invariance in $(i)$ of the above theorem, we mean that for all $\tau_0\geq 0$, $x(\tau_0)\in K$ implies $x(\tau)\in K$ for all $\tau\geq\tau_0$. Note that $(i)$ immediately implies $(ii)$. Furthermore, for autonomous systems, $(i)$ and $(ii)$ are equivalent, because a compact set that all solutions enter can be enlarged to a compact and forward invariant set. However, for non-autonomous systems, $(i)$ is a much stronger property than $(ii)$.


\section{Lemmata} \label{sec:lemmata}


In this section, we collect a series of lemmata, the final two, \cref{lem:unif_boundedness,lem:unif_repel_origin} will be used in \cref{sec:proof_main} to prove \cref{thm:main}.

We start by an elementary lemma. For an equivalent formulation (along with its proof), see \cite[Lemma 7]{boros:2019}. The origin of the latter is \cite[Lemma 3.2]{deng:feinberg:jones:nachman:2011}.
\begin{lemma} \label{lem:elementary}
Let $p \geq 1$ and $w_0 = 1$. Then
\begin{align*}
\sup_{0<w_1,\ldots,w_{p-1} \leq 1}\left(\sum_{i=1}^p w_{i-1} \log w_i\right) \to -\infty \text{ as } w_p\to 0.
\end{align*}
\end{lemma}


The following lemma (that builds on \cref{lem:elementary}) is one of the key observations. It is analogous to \cite[Lemma 8]{boros:2019} and \cite[Lemma 3.3]{deng:feinberg:jones:nachman:2011}.
\begin{lemma} \label{lem:negative}
Let $(\mathcal{X},\mathcal{C},\mathcal{R},\kappa(\tau))$ be a weakly reversible, single linkage class mass-action system with bounded kinetics. Let us denote the relative interior of the $(m-1)$-dimensional simplex in $\mathbb{R}^m_{\geq0}$ by $\Delta^\circ$, i.e.,
\begin{align*}
\Delta^\circ = \{z\in\mathbb{R}^m_+~|~z_1 + \cdots + z_m = 1\}.
\end{align*}
Further, let us define the function $g:[0,\infty)\times \Delta^\circ \to \mathbb{R}$ by
\begin{align*}
g(\tau,z) = \sum_{(i,j)\in\mathcal{R}}\kappa_{ij}(\tau)z_i(\log z_j-\log z_i)\text{ for } (\tau,z) \in [0,\infty)\times \Delta^\circ.
\end{align*}
Then for all $K>0$ there exists a $\delta > 0$ such that
\begin{align*}
g(\tau,z) \leq -K \text{ whenever }\min z\leq\delta.
\end{align*}
\end{lemma}
\begin{proof}
Since
\begin{align*}
g(\tau,z)=\max z \sum_{(i,j)\in\mathcal{R}} \kappa_{ij}(\tau)\frac{z_i}{\max z}\left[\log \frac{z_j}{\max z}-\log \frac{z_i}{\max z}\right],
\end{align*}
for any $\mathcal{R}' \subseteq \mathcal{R}$ we have
\begin{align*}
g(\tau,z) = (\max z)(A + B + C),
\end{align*}
where
\begin{align*}
A &= \sum_{(i,j)\in\mathcal{R}'} \kappa_{ij}(\tau)\frac{z_i}{\max z}\log \frac{z_j}{\max z}, \\
B &= \sum_{(i,j)\in\mathcal{R}\setminus\mathcal{R}'}  \kappa_{ij}(\tau)\frac{z_i}{\max z}\log \frac{z_j}{\max z}, \\
C &= \sum_{(i,j)\in\mathcal{R}}  \kappa_{ij}(\tau)\frac{z_i}{\max z}\left(-\log \frac{z_i}{\max z}\right).
\end{align*}
Since $\max z \geq \frac{1}{m}$, our goal is to show that for each $K>0$, if $\delta>0$ is small enough then for each $z \in \Delta^\circ$ with $\min z\leq\delta$, one can choose $\mathcal{R}'$ such that $A+B+C\leq -mK$. Clearly, $B \leq 0$, because each term in the sum is nonpositive. Also, $C \leq \frac{|\mathcal{R}|}{\varepsilon}$, because $x(-\log x)\leq 1$ for all $0<x\leq1$. To estimate $A$ from above, we will use \cref{lem:elementary}.

Let $w_0 = 1$ and for fixed $p \in \{1,2,\ldots,m-1\}$, let $0<a_p<1$ be such that
\begin{align*}
\sum_{i=1}^p w_{i-1}\log w_i \leq \frac{-mK - \frac{|\mathcal{R}|}{\varepsilon}}{\varepsilon} \text{ for all }0<w_1,\ldots,w_{p-1}\leq 1 \text{ and for all } 0<w_p\leq a_p.
\end{align*}
Finally, let $\delta = \frac{1}{m}\min(a_1,\ldots,a_{m-1})$. We will show in the rest of this proof that $A+B+C\leq -mK$ for all $z \in \Delta^\circ$ with $\min z\leq \delta$.

Fix $z \in \Delta^\circ$ with $\min z \leq \delta$. Then there exist $k$ and $l$ such that $z_k = \max z$ and $z_l = \min z$. The digraph $(\mathcal{C},\mathcal{R})$ is assumed to be strongly connected, therefore, there exists a directed path from $k$ to $l$. Let $\mathcal{R}'$ be the edge set of this directed path, denote by $p\geq1$ the length of this path, and let $\pi(0),\pi(1),\ldots,\pi(p)$ be the enumeration of the vertices visited while travelling from $k$ to $l$ (thus $\pi(0)=k$ and $\pi(p)=l$). Further, let $w_i = \frac{z_{\pi(i)}}{\max z}$ ($i = 0,1, \ldots,p$). With this, $w_0 = 1$, $0<w_1,\ldots,w_p\leq 1$, and
\begin{align*}
A = \sum_{i=1}^p \kappa_{ij}(\tau)w_{i-1}\log w_i \leq \varepsilon\sum_{i=1}^p w_{i-1}\log w_i.
\end{align*}
Then, since $0<w_p = \frac{\min z}{\max z}\leq m\delta \leq a_p$, we have
\begin{align*}
A + B + C \leq \varepsilon\frac{-mK - \frac{|\mathcal{R}|}{\varepsilon}}{\varepsilon} + 0 + \frac{|\mathcal{R}|}{\varepsilon} = -mK.
\end{align*}
This concludes the proof.
\end{proof}


The following lemma, despite its simplicity, plays a crucial role in the sequel. Several versions appear in the CRNT literature, starting with \cite[Lemma 4B]{horn:jackson:1972}. See \cite[Appendix A]{boros:2019} for a list of further appearances.
\begin{lemma} \label{lem:birch}
Let $(\mathcal{X},\mathcal{C},\mathcal{R})$ be a reaction network and $\mathcal{R}'\subseteq\mathcal{R}$ be such that the set of vectors $\{y_j - y_i~|~(i,j) \in \mathcal{R}'\}$ forms a basis of $\mathcal{S}$, hence, $|\mathcal{R}'|=\dim \mathcal{S}$. For a fixed positive stoichiometric class $\mathcal{P}$, let us define $\Psi:\mathcal{P} \to \mathbb{R}^{\dim\mathcal{S}}_+$ by
\begin{align*}
\Psi(x) = \left(\frac{x^{y_j}}{x^{y_i}}\right)_{(i,j) \in \mathcal{R}'} \text{ for }x \in \mathcal{P}.
\end{align*}
Then $\Psi$ is a homeomorphism between $\mathcal{P}$ and $\mathbb{R}^{\dim\mathcal{S}}_+$.
\end{lemma}
\begin{proof}
Since $\frac{x^{y_j}}{x^{y_i}} = e^{\langle y_j - y_i, \log x \rangle}$, the bijectivity of the map $\Psi$ is equivalent to the bijectivity of the map $\log \circ \Psi:\mathcal{P}\to\mathbb{R}^{\dim\mathcal{S}}$, and the latter is equivalent to the following statement. For each $x^* \in \mathbb{R}^n$ there exists a unique $x \in \mathcal{P}$ such that $\log x - \log x^* \in \mathcal{S}^\perp$. The latter statement is widely known to be true, and thus, the bijectivity of $\Psi$ follows. It is also immediate that $\Psi$ is even a homeomorphism.
\end{proof}

Next, we prove that close enough to the boundary of a positive stoichiometric class, there exists a monomial whose value is significantly smaller than the sum of the values of the other monomials.
\begin{lemma} \label{lem:U_M}
Let $(\mathcal{X},\mathcal{C},\mathcal{R})$ be a reaction network with a single linkage class, $\mathcal{P}$ a positive stoichiometric class, and $\Psi:\mathcal{P}\to\mathbb{R}^{\dim\mathcal{S}}_+$ the homeomorphism in \cref{lem:birch}. Further, for $i \in \mathcal{C}$ let $z_i$ denote the quantity $\frac{x^{y_i}}{\sum_{k\in\mathcal{C}}x^{y_k}}$. Finally, for an $M>1$ let $U_M$ denote the cube $\left[\frac{1}{M},M\right]^{\dim\mathcal{S}}$. Then for all $\delta>0$ there exists an $M>1$ such that if $x\in\mathcal{P}\setminus \Psi^{-1}(U_M)$ then $\min z \leq \delta$.
\end{lemma}
\begin{proof}
For $\delta>0$ let $M>1$ so big that $\frac{1}{M+1}\leq\delta$ holds and take $x\in\mathcal{P}\setminus \Psi^{-1}(U_M)$. Then there exists $i,j\in\mathcal{C}$ such that $\frac{z_j}{z_i}\geq M$. Since $z_1+\cdots+z_m=1$, it follows that $z_i \leq \frac{1}{M+1}$. Thus, $\min z \leq z_i \leq\frac{1}{M+1}\leq\delta$. This concludes the proof.
\end{proof}

We now prove that the classical Horn-Jackson function is a Lyapunov function around infinity. Thus, we reprove \cite[Theorem 3.12]{anderson:2011b}.
\begin{lemma} \label{lem:boundedness}
Let $(\mathcal{X},\mathcal{C},\mathcal{R},\kappa(\tau))$ be a weakly reversible, single linkage class mass-action system with bounded kinetics and fix a positive stoichiometric class $\mathcal{P}$. Let us define $V:\mathbb{R}^n_{\geq0}\to\mathbb{R}_{\geq0}$ by
\begin{align*}
V(x) = \sum_{s=1}^n x_s(\log x_s -1)+1 \text{ for } x \in \mathbb{R}^n_{\geq0}.
\end{align*}
Then
\begin{itemize}
\item[(i)] there exists a $K>0$ such that $\frac{\diff}{\diff \tau}V(x(\tau))<0$ whenever $\tau\geq0$ is such that $x(\tau)\in\mathcal{P}$ and $\max x(\tau)\geq K$,
\item[(ii)] all solutions are bounded.
\end{itemize}
\end{lemma}
\begin{proof}
In case $\mathcal{P}$ is bounded, both \emph{(i)} and \emph{(ii)} hold trivially. Assume for the rest of this proof that $\mathcal{P}$ is unbounded.

Statement \emph{(ii)} follows from \emph{(i)}. Indeed, take any such sublevel set of $V$ that contains both the initial condition and the cube $(0,K]^n$. The solution will stay in this bounded sublevel set for all $\tau\geq0$.

It remains to show \emph{(i)}. Note that
\begin{align*}
&\frac{\diff}{\diff \tau}V(x(\tau))=\\
&=\left\langle \log x(\tau), \sum_{(i,j)\in\mathcal{R}} \kappa_{ij}(\tau)x(\tau)^{y_i}(y_j - y_i)\right\rangle =\\
&= \sum_{(i,j)\in\mathcal{R}} \kappa_{ij}(\tau)x(\tau)^{y_i}(\log x(\tau)^{y_j} - \log x(\tau)^{y_i})= \\
&=\left(\sum_{k\in\mathcal{C}}x(\tau)^{y_k}\right)\sum_{(i,j)\in\mathcal{R}} \kappa_{ij}(\tau)\frac{x(\tau)^{y_i}}{\sum_{k\in\mathcal{C}}x(\tau)^{y_k}}\left(\log \frac{x(\tau)^{y_j}}{\sum_{k\in\mathcal{C}}x(\tau)^{y_k}} - \log \frac{x(\tau)^{y_i}}{\sum_{k\in\mathcal{C}}x(\tau)^{y_k}}\right)=\\
&=\left(\sum_{k\in\mathcal{C}}x(\tau)^{y_k}\right)\underbrace{\sum_{(i,j)\in\mathcal{R}} \kappa_{ij}(\tau)z_i(\tau)(\log z_j(\tau) - \log z_i(\tau))}_{g(\tau,z(\tau))},
\end{align*}
where $z_i(\tau)$ is a shorthand for $\frac{x(\tau)^{y_i}}{\sum_{k\in\mathcal{C}}x(\tau)^{y_k}}$. By applying \cref{lem:negative}, let $\delta>0$ be so small that $g(\tau,z(\tau))<0$ whenever $\min z \leq \delta$. Let $\Psi:\mathcal{P} \to \mathbb{R}^{\dim\mathcal{S}}_+$ be the homeomorphism in \cref{lem:birch} and, by applying \cref{lem:U_M}, let $M>1$ be so big that if $x\in\mathcal{P}\setminus \Psi^{-1}(U_M)$ then $\min z \leq \delta$. To conclude the proof of \emph{(i)}, choose $K>0$ so big that the cube $(0,K]^n$ covers the compact set $\Psi^{-1}(U_M)$.
\end{proof}

In a similar way one can show that $V$ is a Lyapunov function near the origin. For later use, we prove a more general result. We show that appropriate variants of the classical Horn-Jackson function are Lyapunov functions near the boundary of the positive orthant. This latter fact appeared implicitly in \cite{anderson:2011a} for weakly reversible mass-action systems with a single linkage class and in \cite{gopalkrishnan:miller:shiu:2014} and \cite{anderson:cappelletti:kim:nguyen:2018} for the larger class of strongly endotactic mass-action systems.
\begin{lemma} \label{lem:V_W}
Let $(\mathcal{X},\mathcal{C},\mathcal{R},\kappa(\tau))$ be a weakly reversible, single linkage class mass-action system with bounded kinetics. Fix a positive stoichiometric class $\mathcal{P}$ and an $\overline{x} \in \overline{\mathcal{P}} \cap \partial\mathbb{R}^n_{\geq0}$. Further, let $W = \{s \in \{1,\ldots,n\}~|~\overline{x}_s=0\}$ and define the function $V_W:\overline{\mathcal{P}}\to\mathbb{R}_{\geq0}$ by
\begin{align*}
V_W(x) = \sum_{s\in W} x_s(\log x_s -1)+1 \text{ for } x \in \overline{\mathcal{P}}.
\end{align*}
Then there exists a neighborhood $U$ of $\overline{x}$ in $\overline{\mathcal{P}}$ such that
\begin{align*}
\frac{\diff}{\diff \tau}V_W(x(\tau))<0 \text{ whenever } \tau\geq0 \text{ is such that } x(\tau) \in U\cap\mathcal{P}.
\end{align*}
\end{lemma}
\begin{proof}
Note that
\begin{align*}
\frac{\diff}{\diff x_s}V_W(x) = \begin{cases} \log x_s & \text{ if } s \in W, \\ 0 & \text{ if } s \in W^c. \end{cases}
\end{align*}
Thus,
\begin{align*}
&\frac{\diff}{\diff \tau}V_W(x(\tau))=\\
&= \sum_{(i,j)\in\mathcal{R}} \kappa_{ij}(\tau)x_{W^c}(\tau)^{y^i_{W^c}}x_W(\tau)^{y^i_W}(\log x_W(\tau)^{y^j_W} - \log x_W(\tau)^{y^i_W})= \\
&=\left(\sum_{k\in\mathcal{C}}x_W(\tau)^{y^k_W}\right)\underbrace{\sum_{(i,j)\in\mathcal{R}} \overline{\kappa}_{ij}(\tau)z_i(\tau)(\log z_j(\tau) - \log z_i(\tau))}_{g(\tau,z(\tau))},
\end{align*}
where $\overline{\kappa}_{ij}(\tau) = \kappa_{ij}(\tau)x_{W^c}(\tau)^{y^i_{W^c}}$ and $z_i(\tau)$ is a shorthand for $\frac{x_W(\tau)^{y^i_W}}{\sum_{k\in\mathcal{C}}x_W(\tau)^{y^k_W}}$.

Fix $\omega>0$ such that $\overline{x}_s-\omega>0$ for all $s \in W^c$. Let $\overline{\varepsilon}>0$ be so small that
\begin{align*}
\overline{\varepsilon} \leq \varepsilon (x_{W^c})^{y^i_{W^c}} \text{ and }\frac{1}{\varepsilon}(x_{W^c})^{y^i_{W^c}}\leq \frac{1}{\overline{\varepsilon}}
\end{align*}
for all $i\in \mathcal{C}$ and for all $x \in \mathbb{R}^n_+$ with the property that $x_{W^c} \in \times_{s\in W^c}[\overline{x}_s-\omega,\overline{x}_s+\omega]$. Then $\overline{\varepsilon} \leq \overline{\kappa}_{ij}(\tau) \leq \frac{1}{\overline{\varepsilon}}$ whenever $\tau\geq0$ is such that $\overline{x}_s - \omega \leq x_s(\tau) \leq \overline{x}_s + \omega$ for all $s \in W^c$.

We conclude the proof similarly to the proof of \cref{lem:boundedness} -- with the subtle difference that we consider the \emph{projected} mass-action system. For the projected network, the set of species is $W$, the set of complexes is $\{y_W~|~y\in\mathcal{C}\}$, and the set of reactions is $\mathcal{R}$. Since $W$ is not an arbitrary nonempty subset of $\mathcal{X}$, but the complement of the support of a point in $\overline{\mathcal{P}}\cap\partial\mathbb{R}^n_{\geq0}$, cannot all the complexes in the projected network be identical. The rate constants of the projected mass-action system are $\overline{\kappa}(\tau)$ as defined above. Although $\overline{\kappa}(\tau)$ is not necessarily bounded away from zero and infinity for all $\tau\geq0$, it is between two bounds whenever the solution is close to $\overline{x}$, and this is enough. Our goal is to show that $V_W$ serves as a Lyapunov function around the origin for the projected mass-action system.

In what follows, the index $W$ refers to the corresponding quantity of the projected mass-action system. Thus, $\mathcal{P}_W$ is the positive stoichiometric class whose closure contains the origin of $\mathbb{R}^W$. By applying \cref{lem:negative}, let $\delta>0$ be so small that $g(\tau,z(\tau))<0$ whenever $\min z(\tau) \leq \delta$. Let $\Psi_W:\mathcal{P}_W \to \mathbb{R}^{\dim \mathcal{P}_W}_+$ be the homeomorphism in \cref{lem:birch} and, by applying \cref{lem:U_M}, let $M>1$ be so big that if $x\in\mathcal{P}_W\setminus \Psi_W^{-1}(U_M)$ then $\min z \leq \delta$. To conclude the proof, choose a neighborhood $U$ of $\overline{x}$ in $\overline{\mathcal{P}}$ such that $|x_s-\overline{x}_s|\leq\omega$ for all $s \in W^c$ and the projection of $U$ to the species $W$ is disjoint from the compact set $\Psi_W^{-1}(U_M)$.
\end{proof}


The contents of next two lemmata are that the trajectories are uniformly bounded and the origin is uniformly repelling, respectively. The proofs are analogous, we detail only the second one.
\begin{lemma} \label{lem:unif_boundedness}
Let $(\mathcal{X},\mathcal{C},\mathcal{R},\kappa(\tau))$ be a weakly reversible, single linkage class mass-action system with bounded kinetics and unbounded positive stoichiometric classes. Fix a positive stoichiometric class $\mathcal{P}$ and let $V:\mathbb{R}^n_{\geq0}\to\mathbb{R}_{\geq0}$ be as in \cref{lem:boundedness}. For a $c>n$, let $U^{>c} = \{x \in \mathcal{P} ~|~ V(x)>c\}$ and $U^{=c} = \{x \in \mathcal{P} ~|~ V(x)=c\}$. Then there exists a $c>n$ such that
\begin{itemize}
\item[(i)] $\frac{\diff}{\diff \tau}V(x(\tau))<0$ whenever $\tau\geq0$ is such that $x(\tau) \in U^{>c}$,
\item[(ii)] the set $\mathcal{P}\setminus U^{>c}$ is bounded and forward invariant,
\item[(iii)] for all $\tau\geq0$ with $x(\tau) \in U^{>c}$, there exists a $\tau'\geq\tau$ such that $x(\tau')\in U^{=c}$.
\end{itemize}
\end{lemma}

\begin{lemma} \label{lem:unif_repel_origin}
Let $(\mathcal{X},\mathcal{C},\mathcal{R},\kappa(\tau))$ be a weakly reversible, single linkage class mass-action system with bounded kinetics. Consider the positive stoichiometric class $\mathcal{P}$ for which $0\in\overline{\mathcal{P}}$ (if such a $\mathcal{P}$ exists). Let $V:\mathbb{R}^n_{\geq0}\to\mathbb{R}_{\geq0}$ be as in \cref{lem:boundedness}. For a $c\geq0$, let
\begin{align*}
U^{>c}     &= \{x \in \mathcal{P} ~|~ V(x)>c     \text{ and }\max x\leq1\},\\
U^{\geq c} &= \{x \in \mathcal{P} ~|~ V(x)\geq c \text{ and }\max x\leq1\},\\
U^{=c}     &= \{x \in \mathcal{P} ~|~ V(x)=c     \text{ and }\max x\leq1\}.
\end{align*}
Then there exists a $c\geq0$ such that
\begin{itemize}
\item[(i)] $U^{>c}$ is nonempty,
\item[(ii)] $0 \in \overline{U^{>c}}$,
\item[(iii)] $\frac{\diff}{\diff \tau}V(x(\tau))<0$ whenever $\tau\geq0$ is such that $x(\tau) \in U^{\geq c}$,
\item[(iv)] the set $\mathcal{P}\setminus U^{>c}$ is forward invariant,
\item[(v)] for all $\tau_0\geq0$ with $x(\tau_0) \in U^{>c}$, there exists a $\tau\geq\tau_0$ such that $x(\tau)\in U^{=c}$.
\end{itemize}
\end{lemma}
\begin{proof}
Since $V$ has a local maximum at $0$ and $V(0) = n$, the requirements $(i)$ and $(ii)$ are satisfied for all $c<n$. If $c<n$ is close enough to $n$ then $(iii)$ follows from \cref{lem:V_W}. The property $(iv)$ follows immediately from $(iii)$.

To prove $(v)$, let $\tau_0\geq0$ with $x(\tau_0) \in U^{>c}$, where $c<n$ is such that $(i)$, $(ii)$, $(iii)$, and $(iv)$ hold. Assume by contradiction that there exists an $\overline{x} \in \overline{\mathcal{P}}\cap\partial\mathbb{R}^n_{\geq0}$ and $\tau_0\leq \tau_1 < \tau_2 < \cdots$ such that $x(\tau_k)\in U^{>c}$ converges to $\overline{x}$ as $k\to\infty$. Let $W=\{s \in \{1,\ldots,n\}~|~\overline{x}_s=0\}$. W.l.o.g., choose the accumulation point $\overline{x}$ such that $W$ is maximal. Note that $W \neq \{1,\ldots,n\}$ by $(iii)$. Since $V_W$ is continuous and $V_W(\overline{x}) > V_W(x)$ for all $x\in U^{>c}$, we can choose a subsequence, again denoted by $(\tau_k)_{k\geq1}$, such that $V_W(x(\tau_k))$ increases monotonically to $V_W(\overline{x})$ as $k\to \infty$. Then for each $k\geq 1$ there exists a $\tau_k'$ such that $\tau_k<\tau_k'<\tau_{k+1}$, $V_W(x(\tau_k))\leq V_W(x(\tau_k'))\leq V_W(\overline{x})$, and $\frac{\diff}{\diff \tau}V_W(x(\tau_k'))\geq0$. Now pick a subsequence, again denoted by $(\tau_k')_{k\geq1}$, such that $x(\tau_k')$ converges to some point $x'$ as $k\to \infty$. Because $V_W(\overline{x}) = \lim_{k\to\infty} V_W(x(\tau_k')) = V_W(x')$, we have $x'_s=0$ if $s\in W$. In fact, because of the maximality of $W$, $x'_s=0$ if and only if $s\in W$. By \cref{lem:V_W}, applied to $x'$, we get that $\frac{\diff}{\diff \tau}V_W(x(\tau_k'))<0$ for large enough $k$, a contradiction.

Therefore, the trajectory has no accumulation points in $\overline{\mathcal{P}}\cap\partial\mathbb{R}^n_{\geq0}$. Now assume by contradiction that for all $\tau\geq\tau_0$ we have $x(\tau)\in U^{>c}$. Thus, the closure of the trajectory $(x(\tau))_{\tau\geq\tau_0}$, is a compact set disjoint from the boundary of $\mathbb{R}^n_{\geq0}$. Hence, by \cref{lem:V_W}, there exists an $L>0$ such that $\frac{\diff}{\diff \tau}V(x(\tau))<-L$ whenever $\tau\geq\tau_0$. This guarantees that the solution reaches $U^{=c}$ in finite time, a contradiction. This concludes the proof of $(v)$.
\end{proof}

\section{Proof of \cref{thm:main}} \label{sec:proof_main}

In this section, we prove statement $(i)$ of \cref{thm:main}. As we already mentioned, statement $(ii)$ then follows immediately.

Fix a positive stoichiometric class $\mathcal{P}$. We will construct a compact and forward invariant set $K\subseteq\mathcal{P}$ with the property that each solution starting in $\mathcal{P}$ enters $K$ in finite time. The boundary of $\mathcal{P}$ is composed of level set pieces of $V$ and $V_W$ in \cref{lem:boundedness,lem:V_W}, respectively.

If $\mathcal{P}$ is unbounded, pick a $c>n$ as in \cref{lem:unif_boundedness}, and let $U_\infty=U^{>c}$. If $\mathcal{P}$ is bounded, let $U_\infty=\emptyset$.

Pick a vertex $\overline{x}$ of $\overline{\mathcal{P}}$ and let $W = \{s\in\{1,\ldots,n\}~|~\overline{x}_s=0\}$. Our goal is to define a set $U_W\subseteq\mathcal{P}$ with
\begin{itemize}
\item[(i)] $U_W$ is open relative to $\mathcal{P}$,
\item[(ii)] $\overline{U_W}$ is a neighborhood of $\overline{x}$ in $\overline{\mathcal{P}}$,
\item[(iii)] $\mathcal{P}\setminus U_W$ is forward invariant.
\item[(iv)] for all $\tau\geq0$ with $x(\tau)\in U_W$, there exists a $\tau'\geq\tau$ such that $x(\tau')\in \mathcal{P}\setminus U_W$.
\end{itemize}
To this end, as in the proof of \cref{lem:unif_repel_origin}, we consider the \emph{projected} mass-action system. For the projected network, the set of species is $W$, the set of complexes is $\{y_W~|~y\in\mathcal{C}\}$, and the set of reactions is $\mathcal{R}$. Since $W$ is not an arbitrary nonempty subset of $\mathcal{X}$, but the complement of the support of a point in $\overline{\mathcal{P}}\cap\partial\mathbb{R}^n_{\geq0}$, not all the complexes in the projected network can be identical. Let $\mathcal{P}_W$ be the positive stoichiometric class in the projected network, whose closure contains the origin of $\mathbb{R}^W$. The rate constants of the projected mass-action system are $\overline{\kappa}_{ij}(\tau) = \kappa_{ij}(\tau)x_{W^c}(\tau)^{y^i_{W^c}}$ for $(i,j)\in\mathcal{R}$. By \cref{lem:boundedness}, $\overline{\kappa}(\tau)$ is bounded away from infinity. And, though $\overline{\kappa}(\tau)$ is not necessarily bounded away from zero, it is between two bounds whenever the solution of the original system is close to $\overline{x}$.

Now apply \cref{lem:unif_repel_origin} to the projected system and pick a $c\geq0$ such that $U^{>c}\subseteq\mathcal{P}_W$ has the five properties in that lemma. Then the set $U_W=\{x\in\mathcal{P}~|~x_W\in U^{>c}\}$ satisfies the four requirements above.

Denote by $\mathcal{W}_0$ the set of those $W$, which correspond to vertices of $\overline{\mathcal{P}}$. Let
\begin{align*}
\mathcal{P}_1=\mathcal{P} \setminus \left(U_\infty \cup \bigcup_{W\in\mathcal{W}_0} U_W\right).
\end{align*}

Pick now an edge $e$ of $\overline{\mathcal{P}}$ and let $\overline{e} = e \cap \mathcal{P}_1$. Pick an $\overline{x}\in\overline{e}$ and let $W = \{s\in\{1,\ldots,n\}~|~\overline{x}_s=0\}$ (note that $W$ does not depend on the choice of $\overline{x}$). Our goal is to define a set $U_W\subseteq\mathcal{P}$ with
\begin{itemize}
\item[(i)] $U_W$ is open relative to $\mathcal{P}$,
\item[(ii)] $\overline{U_W}$ is a neighborhood of $\overline{e}$ in $\overline{\mathcal{P}}$,
\item[(iii)] $\mathcal{P}_1\setminus U_W$ is forward invariant.
\item[(iv)] for all $\tau\geq0$ with $x(\tau)\in U_W$, there exists a $\tau'\geq\tau$ such that $x(\tau')\in \mathcal{P}_1\setminus U_W$.
\end{itemize}
To this end, we consider the \emph{projected} mass-action system, similarly as above. In a neighborhood of $\overline{e}$, the rate constants of the projected mass-action system are bounded away from zero, and thus, \cref{lem:unif_repel_origin} applies to the projected system. Similarly as above, we get the set $U_W$ that satisfies the four requirements.

Denote by $\mathcal{W}_1$ the set of those $W$, which correspond to edges of $\overline{\mathcal{P}}$. Let
\begin{align*}
\mathcal{P}_2=\mathcal{P}_1 \setminus \left(\bigcup_{W\in\mathcal{W}_1} U_W\right).
\end{align*}

Continue this procedure with the two-dimensional faces of $\mathcal{P}$ to get $\mathcal{P}_3$, etc. Finally, the last step produces the compact and forward invariant set $K=\mathcal{P}_{\dim\mathcal{S}}$ in $\mathcal{P}$ that attracts every solution in finite time. This concludes the proof of \cref{thm:main}.


\section{Examples}
\label{sec:examples}

In this section we give examples with two linkage classes that illustrate the limitations of the usage of $V$ (or its variant) as a Lyapunov function.

\subsection{An example with two linkage classes, where $V$ does not work at the origin}\label{subsec:example_origin}

Consider the (time-independent) mass-action system
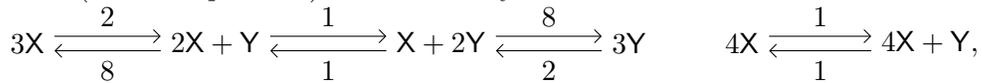
\begin{center}
\begin{tikzpicture}
\node (A1) at (0,0) {$3\mathsf{X}$};
\node (A2) at (2.5,0) {$2\mathsf{X} + \mathsf{Y}$};
\node (A3) at (5.5,0) {$ \mathsf{X} + 2\mathsf{Y}$};
\node (A4) at (8,0) {$3\mathsf{Y}$};
\draw[arrows={-angle 90},transform canvas={yshift= .25em}] (A1) to node [above] {$2$} (A2);
\draw[arrows={-angle 90},transform canvas={yshift=-.25em}] (A2) to node [below] {$8$} (A1);
\draw[arrows={-angle 90},transform canvas={yshift= .25em}] (A2) to node [above] {$1$} (A3);
\draw[arrows={-angle 90},transform canvas={yshift=-.25em}] (A3) to node [below] {$1$} (A2);
\draw[arrows={-angle 90},transform canvas={yshift= .25em}] (A3) to node [above] {$8$} (A4);
\draw[arrows={-angle 90},transform canvas={yshift=-.25em}] (A4) to node [below] {$2$} (A3);

\node (A5) at (9.5,0) {$4\mathsf{X}$};
\node (A6) at (12,0) {$4\mathsf{X} + \mathsf{Y}$,};
\draw[arrows={-angle 90},transform canvas={yshift= .25em}] (A5) to node [above] {$1$} (A6);
\draw[arrows={-angle 90},transform canvas={yshift=-.25em}] (A6) to node [below] {$1$} (A5);
\end{tikzpicture}
\end{center}
which results in the ODE
\begin{align*}
\dot{x} &= f_1(x,y) = -2x^3 + 7 x^2y -7 xy^2 + 2y^3, \\
\dot{y} &= f_2(x,y) = +2x^3 - 7 x^2y +7 xy^2 - 2y^3 + x^4 - x^4y.\\
\end{align*}

For $(x^*,y^*)\in\mathbb{R}^2_+$, let us define $V_{(x^*,y^*)}:\mathbb{R}^2_{\geq0}\to\mathbb{R}_{\geq0}$, the Horn-Jackson function centered around $(x^*,y^*)$, by
\begin{align*}
V_{(x^*,y^*)}(x,y) = x\left(\log \frac{x}{x^*} -1\right)+x^* + y\left(\log \frac{y}{y^*} -1\right)+y^* \text{ for } (x,y) \in \mathbb{R}^2_{\geq0}.
\end{align*}
Then the time derivative of $V_{(x^*,y^*)}$ along trajectories is
\begin{align*}
&(\grad V_{(x^*,y^*)})(x,y) \left[\begin{array}{c}f_1(x,y)\\f_2(x,y)\end{array}\right] =
\left[\log \frac{x}{x^*},\log \frac{y}{y^*}\right] \left[\begin{array}{c}f_1(x,y)\\f_2(x,y)\end{array}\right] =\\
&=-\left(y-\frac{x}{2}\right)(y-x)(y-2x) \left(\log\frac{y}{y^*} - \log\frac{x}{x^*}\right) + x^4(1-y)\log\frac{y}{y^*}.
\end{align*}
Thus, for each $(x^*,y^*)\in\mathbb{R}^2_+$, every neighborhood of the origin contains $(x,y)\in\mathbb{R}^2_+$ for which the above expression is positive. Thus, the function $V_{(x^*,y^*)}$ cannot serve as a Lyapunov function close to the origin.

On the other hand, since every planar reversible mass-action system is permanent (see \cite{simon:1995}), the example of this subsection is permanent.

\subsection{An example with two linkage classes, where $V$ does not work at infinity}

Consider the (time-independent) mass-action system
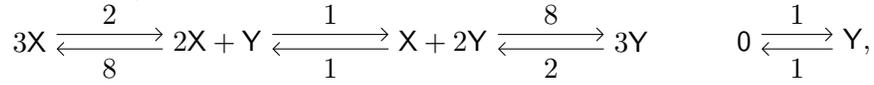
\begin{center}
\begin{tikzpicture}
\node (A1) at (0,0) {$3\mathsf{X}$};
\node (A2) at (2.5,0) {$2\mathsf{X} + \mathsf{Y}$};
\node (A3) at (5.5,0) {$ \mathsf{X} + 2\mathsf{Y}$};
\node (A4) at (8,0) {$3\mathsf{Y}$};
\draw[arrows={-angle 90},transform canvas={yshift= .25em}] (A1) to node [above] {$2$} (A2);
\draw[arrows={-angle 90},transform canvas={yshift=-.25em}] (A2) to node [below] {$8$} (A1);
\draw[arrows={-angle 90},transform canvas={yshift= .25em}] (A2) to node [above] {$1$} (A3);
\draw[arrows={-angle 90},transform canvas={yshift=-.25em}] (A3) to node [below] {$1$} (A2);
\draw[arrows={-angle 90},transform canvas={yshift= .25em}] (A3) to node [above] {$8$} (A4);
\draw[arrows={-angle 90},transform canvas={yshift=-.25em}] (A4) to node [below] {$2$} (A3);

\node (A5) at (9.5,0) {$\mathsf{0}$};
\node (A6) at (11,0) {$\mathsf{Y}$,};
\draw[arrows={-angle 90},transform canvas={yshift= .25em}] (A5) to node [above] {$1$} (A6);
\draw[arrows={-angle 90},transform canvas={yshift=-.25em}] (A6) to node [below] {$1$} (A5);
\end{tikzpicture}
\end{center}
which results in the ODE
\begin{align*}
\dot{x} &= f_1(x,y) = -2x^3 + 7 x^2y -7 xy^2 + 2y^3, \\
\dot{y} &= f_2(x,y) = +2x^3 - 7 x^2y +7 xy^2 - 2y^3 + 1 - y.\\
\end{align*}

With $V_{(x^*,y^*)}:\mathbb{R}^2_{\geq0}\to\mathbb{R}_{\geq0}$ being the same function as in \cref{subsec:example_origin}, the time derivative of $V_{(x^*,y^*)}$ along trajectories is
\begin{align*}
&(\grad V_{(x^*,y^*)})(x,y) \left[\begin{array}{c}f_1(x,y)\\f_2(x,y)\end{array}\right] =
\left[\log \frac{x}{x^*},\log \frac{y}{y^*}\right] \left[\begin{array}{c}f_1(x,y)\\f_2(x,y)\end{array}\right] =\\
&=-\left(y-\frac{x}{2}\right)(y-x)(y-2x) \left(\log\frac{y}{y^*} - \log\frac{x}{x^*}\right) + (1-y)\log\frac{y}{y^*}.
\end{align*}
Thus, for each $(x^*,y^*)\in\mathbb{R}^2_+$ and for each $M>0$ there exists $(x,y)$ with $x\geq M$, $y\geq M$ for which the above expression is positive. Thus, the function $V_{(x^*,y^*)}$ cannot serve as a Lyapunov function around infinity.

On the other hand, since every planar reversible mass-action system is permanent (see \cite{simon:1995}), the example of this subsection is permanent.

\subsection{An example, which is complex balanced}

Consider a modification of \cite[Example 9.6]{gopalkrishnan:miller:shiu:2014}, namely, the (time-independent) mass-action system
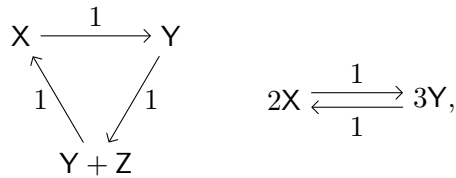
\begin{center}
\begin{tikzpicture}
\node (A1) at (0,0) {$\mathsf{X}$};
\node (A2) at (2,0) {$\mathsf{Y}$};
\node (A3) at (1,-1.7) {$\mathsf{Y} + \mathsf{Z}$};
\draw[arrows={-angle 90}] (A1) to node [above] {$1$} (A2);
\draw[arrows={-angle 90}] (A2) to node [right] {$1$} (A3);
\draw[arrows={-angle 90}] (A3) to node [left] {$1$} (A1);

\node (A4) at (3.5,-0.85) {$2\mathsf{X}$};
\node (A5) at (5.5,-0.85) {$3\mathsf{Y}$,};
\draw[arrows={-angle 90},transform canvas={yshift= .25em}] (A4) to node [above] {$1$} (A5);
\draw[arrows={-angle 90},transform canvas={yshift=-.25em}] (A5) to node [below] {$1$} (A4);
\end{tikzpicture}
\end{center}
which results in the ODE
\begin{align*}
\dot{x} &= f_1(x,y,z) = - x + yz - 2x^2 + 2y^3, \\
\dot{y} &= f_2(x,y,z) = + x - yz + 3x^2 - 3y^3, \\
\dot{z} &= f_3(x,y,z) = + y - yz.
\end{align*}
The point $(1,1,1)$ is the unique positive equilibrium. Furthermore, this point is a \emph{complex balanced} equilibrium, and thus, according to the \emph{Global Attractor Conjecture} \cite{craciun:dickenstein:shiu:sturmfels:2009}, which is proven for the three species case (see \cite{craciun:nazarov:pantea:2013} and \cite{pantea:2012}), it is globally asymptotically stable, and therefore the system is permanent. However, as we demonstrate below, the method of the present paper fails to prove permanence.

The set of boundary equilibria consists of the points $(0,0,\overline{z})$ with $\overline{z}\geq 0$. For $(x^*,y^*)\in\mathbb{R}^2_+$, let us define $V_{(x^*,y^*)}:\mathbb{R}^3_{\geq0}\to\mathbb{R}_{\geq0}$ by
\begin{align*}
V_{(x^*,y^*)}(x,y,z) = x\left(\log \frac{x}{x^*} -1\right)+x^* + y\left(\log \frac{y}{y^*} -1\right)+y^* \text{ for } (x,y,z) \in \mathbb{R}^3_{\geq0}.
\end{align*}
Then the time derivative of $V_{(x^*,y^*)}$ along trajectories is
\begin{align*}
&(\grad V_{(x^*,y^*)})(x,y,z) \left[\begin{array}{c}f_1(x,y,z)\\f_2(x,y,z)\\f_3(x,y,z)\end{array}\right] =
\left[\log \frac{x}{x^*},\log \frac{y}{y^*},0\right] \left[\begin{array}{c}f_1(x,y,z)\\f_2(x,y,z)\\f_3(x,y,z)\end{array}\right] =\\
&=(x-yz) \left(\log\frac{y}{y^*} - \log\frac{x}{x^*}\right) + (x^2-y^3)\left(\log\left(\frac{y}{y^*}\right)^3 - \log\left(\frac{x}{x^*}\right)^2\right).
\end{align*}
Thus, for each $\overline{z}>0$ and for each $(x^*,y^*)\in\mathbb{R}^2_+$, every neighborhood of $(0,0,\overline{z})$ contains $(x,y,z)\in\mathbb{R}^3_+$ for which the above expression is positive. Thus, the function $V_{(x^*,y^*)}$ cannot serve as a Lyapunov function close to $(0,0,\overline{z})$ with $\overline{z}>0$. The method in \cite{anderson:2011a} therefore cannot be extended in a straightforward manner to systems with multiple linkage classes.

\appendix
\section{The acyclic digraph of the implications} 
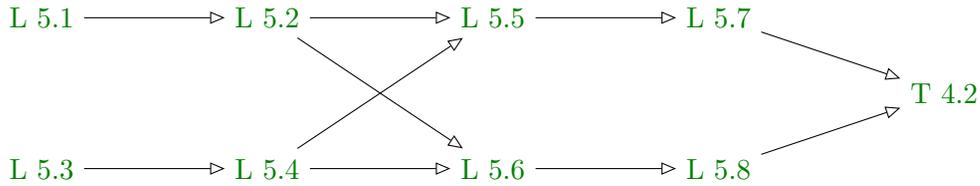
\label{sec:acyclic}

\begin{center}
\begin{tikzpicture}[auto]

\node (A) at ( 0, 0) {\Cref{lem:elementary}};
\node (B) at ( 3, 0) {\Cref{lem:negative}};
\node (C) at ( 0,-2) {\Cref{lem:birch}};
\node (D) at ( 3,-2) {\Cref{lem:U_M}};
\node (E) at ( 6, 0) {\Cref{lem:boundedness}};
\node (F) at ( 6,-2) {\Cref{lem:V_W}};
\node (G) at ( 9, 0) {\Cref{lem:unif_boundedness}};
\node (H) at ( 9,-2) {\Cref{lem:unif_repel_origin}};
\node (I) at (12,-1) {\Cref{thm:main}};

\draw[-{Triangle[open,angle=45:6pt]}] (A) -- (B);
\draw[-{Triangle[open,angle=45:6pt]}] (C) -- (D);
\draw[-{Triangle[open,angle=45:6pt]}] (B) -- (E);
\draw[-{Triangle[open,angle=45:6pt]}] (B) -- (F);
\draw[-{Triangle[open,angle=45:6pt]}] (D) -- (E);
\draw[-{Triangle[open,angle=45:6pt]}] (D) -- (F);
\draw[-{Triangle[open,angle=45:6pt]}] (E) -- (G);
\draw[-{Triangle[open,angle=45:6pt]}] (F) -- (H);
\draw[-{Triangle[open,angle=45:6pt]}] (G) -- (I);
\draw[-{Triangle[open,angle=45:6pt]}] (H) -- (I);

\end{tikzpicture}
\end{center}

\section*{Acknowledgment}
This paper was inspired by discussions with Stefan M\"uller and Georg Regensburger on the stability of complex balanced equilibria.

\bibliographystyle{siamplain}
\bibliography{references}

\end{document}